\documentclass[graybox]{svmult}

\usepackage{mathptmx}       
\usepackage{helvet}         
\usepackage{courier}        
\usepackage{type1cm}        
\usepackage{makeidx}         
\usepackage{graphicx}        
\usepackage{multicol}        
\usepackage[bottom]{footmisc}

\makeindex             
\def\tred{\textcolor{red}}
\def\tred{{}}

\usepackage{amsfonts,amsmath,amssymb}   

\def\Asp{\boldsymbol{A}}
\def\AA{\boldsymbol{A}}

\newcommand{\AspN}{{(\Asp, \, \|\ebbes\|_\Asp)}}     
\def\AspN{{\NSPB \Asp}}  
\newcommand{\ebbes}{\mbox{$\,\cdot\,$}}     
\newcommand\AspP{{\Asp'}}     
\newcommand{\Axb}{ { \AA \ast \xx = \bb }} 
\newcommand{\xx}{ {\bf x}}     
\newcommand{\bb}{ {\bf b}}     
\newcommand{\Bsp}{{\boldsymbol B}}     
\newcommand{\BspN}{(\Bsp, \, \|\ebbes\|_\Bsp)}     
\newcommand{\CORd}{{\COsp(\Rst^d)}}     
\newcommand{\COsp}{{\Csp_{\negthinspace 0}}}     
\newcommand{\Rst}{{\mathbb R}}     
\newcommand{\CORdN}{{\big( \COsp(\Rst^d), \, \|\ebbes\|_\infty \big)}}     
\newcommand{\Csp}{{\boldsymbol C}}     
\newcommand{\CbRd}{{\Cbsp(\Rdst)}}     
\newcommand{\Cbsp}{{\Csp_{\negthinspace b}}}     
\newcommand{\Rdst}{{{\Rst^d}}}     
\newcommand{\CbRdN}{{\big( \Cbsp(\Rst^d), \, \|\ebbes\|_\infty \big)}}     
\newcommand{\CcRd}{{\Ccsp(\Rst^d)}}     
\newcommand{\Ccsp}{{\Csp_{\negthinspace c}}}     
\newcommand{\Cst}{{\mathbb C}}     
\newcommand{\DPRd}{{\Dcsp'(\Rst^d)}}     
\newcommand{\Dcsp}{{\boldsymbol{\mathcal D}}}     
\newcommand{\DRd}{{\Dcsp(\Rst^d)}}     
\newcommand\FLi{{\mathcal F}{\negthinspace \Lisp}}     
\newcommand{\Lisp}{{\Lsp^1}}     
\newcommand{\FLiRd}{{ \FLi(\Rdst) }}     
\newcommand\FLiRdN{\big( \FLiRd, \, \|\ebbes\|_{\FLisp} \big)}     
\newcommand{\FLisp}{{ {\mathcal F} \negthinspace \Lisp}}     
\newcommand{\FT}{{\operatorname{{\mathcal F}}}}     
\newcommand\Hilb{\mathcal H}     
\newcommand\Hs{{{\cal H}_s}}     
\newcommand\HsRd{{{\mathcal H}_s (\Rdst)}}     
\newcommand\HsRdN{{({\cal H}_s (\Rdst), \, \|\ebbes\|_{\Hs} \big)}}     
\newcommand{\Hsp}{{\boldsymbol H}}     
\newcommand{\IFT}{\operatorname{\mathcal F}^{-1}}     
\newcommand{\LiG}{{\Lisp(G)}}     
\newcommand{\LiGN}{\big( \LiG, \, \|\ebbes\|_1 \big)}     
\newcommand{\LiR}{{\Lisp(\Rst)}}     
\newcommand{\LiRN}{\big( \LiR, \, \|\ebbes\|_1 \big)}     
\newcommand{\LiRd}{{\Lisp \nth (\Rst^d)}}     
\newcommand\nth{\negthinspace}     
\newcommand{\LiRdN}{\big( \LiRd, \, \|\ebbes\|_1 \big)}     
\newcommand{\LiRtd}{{\Lisp(\Rst^{2d})}}     
\newcommand\LinfRd{ \Lsp^{\infty}(\Rst^{d})}     
\newcommand{\Lsp}{{\boldsymbol L}}     
\newcommand{\LpRd}{{\Lpsp(\Rst^d)}}     
\newcommand{\Lpsp}{{\Lsp^p}}     
\newcommand{\LpRdN}{\big( \LpRd, \, \|\ebbes\|_p \big)}     
\newcommand{\LtR}{{\Ltsp(\Rst)}}     
\newcommand{\Ltsp}{{\Lsp^2}}     
\newcommand{\LtRN}{\big( \LtR, \, \|\ebbes\|_2 \big)}     
\newcommand{\LtRd}{{\Ltsp(\Rst^d)}}     
\newcommand{\LtRdN}{\big( \LtRd, \, \|\ebbes\|_2 \big)}     
\newcommand{\LtT}{{\Ltsp(\Tst)}}     
\newcommand{\Tst}{\mathbb{T}}     
\newcommand{\MbRd}{{\Mbsp(\Rst^d)}}     
\newcommand{\Mbsp}{{\Msp_{\negthinspace b}}}     
\newcommand{\Msp}{{\boldsymbol M}}     
\newcommand\MiiRd{\Msp^{1,1}_{0}(\Rd)}     
\newcommand\Rd{\Rdst}     
\newcommand{\MininRd}{{\Msp^{\infty,\infty}(\Rdst)}}     
\newcommand{\MpqRd}{{ \Msp^{p,q}(\Rdst)}}     
\newcommand\MpqRdN{\big( \MpqRd, \, \|\ebbes\|_{\Mpqsp} \big)}     
\newcommand{\Mpqsp}{{ \Msp^{p,q}}}     
\newcommand{\Nst}{{\mathbb N}}     
\newcommand\PV{{P_{\Vsp}}}     
\newcommand{\Vsp}{{\boldsymbol V}}     
\newcommand\QRC{{\Qst \subset \Rst \subset \Cst}}     
\newcommand{\Qst}{{\mathbb Q}}     
\newcommand{\Rdsth}{{\widehat{\Rst}^d}}     
\newcommand\SHZd{{\Shah_\Zdst}}     
\newcommand{\Shah}{{\makebox[2.3ex][s]{$\sqcup$\hspace{-0.15em}\hfill $\sqcup$}\, \, }}     
\newcommand{\Zdst}{{\Zst^d}}     
\newcommand{\SO}{\SOsp}     
\newcommand{\SOsp}{{\Ssp_{\negthinspace 0}}}     
\newcommand{\SOGTr}{{ (\SOsp,\Ltsp \negthinspace ,\SOPsp) }}     
\newcommand{\SOPsp}{{\Ssp_{\negthinspace 0}'}}     
\newcommand{\SOGTrRd}{{ (\SOsp,\Ltsp,\SOPsp)(\Rdst) }}     
\newcommand{\SOPRd}{{\SOPsp(\Rst^d)}}     
\newcommand{\SOPRdN}{(\SOPRd , \| \ebbes \|_{\SOPsp} ) }     
\newcommand{\SOPRtd}{{\SOPsp(\Rst^{2d})}}     
\newcommand{\SOPnorm}[1]{{\lVert #1 \rVert_\SOPsp}}     
\newcommand{\Ssp}{{\boldsymbol S}}     
\newcommand{\SORd}{{\SOsp(\Rst^d)}}     
\newcommand{\SORdN}{\big( \SORd, \|\ebbes\|_\SOsp \big)}     
\newcommand{\SOnorm}[1]{{\lVert #1 \rVert_\SOsp}}     
\newcommand{\SOprime}{\textnormal{\textbf{S}}'_{0}}     
\newcommand{\ScPRd}{{\ScPsp(\Rst^d)}}     
\newcommand{\ScPsp}{{\Scsp'}}     
\newcommand{\Scsp}{{\boldsymbol{\mathcal S}}}     
\newcommand{\ScRd}{{\Scsp(\Rst^d)}}     
\newcommand{\TFd}{{{ \Rdst \times \Rdsth }}}     
\newcommand{\WCOliRd}{\WTsp \COsp \lisp (\Rdst) }     
\newcommand{\lisp}{{\lsp^1}}     
\newcommand\WFLiliRd{{\Wsp(\FLi \nth, \lisp)(\Rdst)}}     
\newcommand{\Wsp}{{\boldsymbol W}}     
\newcommand\WLplqRd{{ \Wsp(\Lpsp,\lqsp)(\Rdst) }}     
\newcommand\lqsp{{ \lsp^q}}     
\newcommand{\Zst}{{\mathbb Z}}     
\newcommand{\cG}{\mathscr{G}}     
\newcommand\chck{^\checkmark \negthinspace}     
\newcommand{\ee}{ {\bf e}}     
\newcommand\epso{{ \varepsilon > 0 }}     
\newcommand\espo{{ \varepsilon > 0 }}     
\newcommand\fSO{{f \in \SOsp}}     
\newcommand{\fSORd}{ f \in \SORd }     
\newcommand\fchk{{f^\checkmark}}     
\newcommand{\fhat}{{\widehat{f}}}     
\newcommand\fper{{f_{per}}}     
\newcommand{\gd}{{\widetilde{g}}} 
\newcommand\hatf{{\widehat{f}}}     
\newcommand{\hatg}{{\widehat{g}}}     
\newcommand\hatsi{{\widehat{\sigma}}}     
\newcommand\hkr{\hookrightarrow}     
\newcommand{\intR}{\int_{\Rst}}     
\newcommand{\intRd}{\int_{\Rst^d}}     
\newcommand{\intRtd}{\int_{\Rst^{2d}}}     
\newcommand{\intinf}{\int_{-\infty}^{\, \infty}}     
\newcommand{\kiZd}{{{k \in \Zdst}}}
\newcommand\kin{{_{k=1}^n}}     
\newcommand\limninf{{\lim_{n \to \infty}}}     
\newcommand{\linfnorm}[1]{{\lVert #1 \rVert_{\infty}}}     
\newcommand{\linorm}[1]{{\lVert #1 \rVert_1}}     
\newcommand{\linsp}{{\lsp^\infty}}     
\newcommand{\lsp}{{\boldsymbol\ell}}     
\newcommand{\ltnorm}[1]{{\lVert #1 \rVert_2}}     
\newcommand\ntn{{n \times n}}     
\newcommand{\ofp}[1]{{\slp{#1}\srp}}     
\newcommand{\slp}{{{\raise 0.5pt \hbox{\footnotesize $($}}}}     
\newcommand{\srp}{{{\raise 0.5pt \hbox{\footnotesize $)$}}}}     
\newcommand\pilam{{\pi\ofp{\lambda}}}     
\newcommand\pilamg{ \pi(\lambda) g}     
\newcommand\pilamgd{ \pi(\lambda) \tilde g}     
\newcommand\siSOP{{ \sigma \in \SOPsp }}     
\newcommand\siSOPRd{ \sigma \in \SOPRd }     
\newcommand\sigSOP{{ \sigma \in \SOPsp }}     
\newcommand\signo{ \sigma_0 = \wst-\mbox{lim}_n \sigma_n }     
\newcommand{\sonorm}[1]{{\lVert #1 \rVert_\SOsp}}     
\newcommand{\sumkZd}{\sum_{\kiZd}}     
\newcommand\sumkin{\sum_{k=1}^n}     
\newcommand\sumnZd{\sum_{n\in\Zdst}}     
\newcommand{\supp}{\operatorname{supp}}     
\newcommand\suth{{ \, | \, } }     
\newcommand\veps{{\varepsilon}}     
\newcommand{\wdash}{{ w^* \negthinspace \mbox{-}}}     
\newcommand{\ws}{{w_s}}     
\newcommand\wst{ w^{*} }     
\newcommand\wstd{{w^* \negthinspace -}}     
\newcommand\wstlim{{ \wdash \lim \, }}     
\newcommand{\wwst}{{\wdash \wdash}}     
\newcommand{\yy}{ {\bf y}}     

\def\cG{{G}} 
\def\ws{{weak*}}
\def\HsRd{{{\Hsp}_s (\Rdst)}}
\def\PV{{ \operatorname{P}_\Vsp}}                     

\def\signo{{\sigma_0 = \wst-\mbox{lim}_n \sigma_n}}
\newcommand\siginf{{(\sigma_n)_{n \geq 1}}}

\def\Shah{{\rule{1pt}{1.5ex}\rule{0.5em}{1pt}\rule{1pt}{1.5ex}\rule{0.5em}{1pt}\rule{1pt}{1.5ex}\,}}
\def\SHZd{{\Shah_{\Zdst}}}

\def\SOnorm#1{{\lVert #1 \rVert_\SOsp}}
\def\sonorm#1{{\lVert #1 \rVert_\SOsp}}
\def\SOPnorm#1{{\lVert #1 \rVert_\SOPsp}}
\def\Bnorm#1{{ \| #1 \|_\Bsp }}
\def\WTsp#1#2{{\Wsp(#1,#2)}}
\def\NSPB#1{{\left (#1,\| \ebbes \|_{#1} \right )}}
\def\citeX{\cite}
\def\tred{}
\def\tblue{}

\begin{document}
 \title*{Modelling Signals as Mild Distributions}
\author{Hans G.~Feichtinger}
\institute{Faculty of Mathematics, University Vienna, AUSTRIA * \newline
and ARI (Acoustic Research Institute), OEAW    \newline \email{hans.feichtinger@univie.ac.at} }
%
%
\maketitle 

\newcommand\sigseq{{(\sigma_n)_{n \geq 1}}}

\def\sumjZd{{\sum_{j \in \Zdst}}}



\abstract{\newline This note gives a summary of ideas concerning Applied
Fourier Analysis, mostly formulated for those who have to give such courses
to engineers or mathematicians interested in real life applications.
It tries to answer recurrent questions arising regularly
after my talks on the subject. It outlines alternative ways
of presenting the core material of Fourier Analysis in a way which is
 supposed to  help students
to grasp the relevance of this transform in the context of applications.
Essentially we consider functions in $\SORd$ as possible measurements,
and the elements of the dual space (which can be also described by
various completion procedures) is thus the collection of all ``things''
(in the spirit of signals) which can be measured in a reasonable way.
We call them {\it mild distributions}. In other words, we want to base
{\it signal analysis} on the mathematical theory of mild distributions.
They do not need to be defined
in a pointwise sense. Dirac's Delta or Dirac combs are natural examples.
The are ``as real as point masses in physics''. Various equivalent
approaches (including a sequential approach) help to verify the
relevant results.
\newline
The material is based on the experiences gained by the author
in the last 30 years by teaching the subject in an abstract or
application oriented manner, at different universities, and
to  audiences with quite different background.
The text is also
based on the insight coming from computational work, partially
through numerical simulations in the area of irregular sampling
or Gabor Analysis, but also from applied projects dealing with
real world problems. In addition some ideas have been developed
during the supervision of internships
of (mostly foreign) students  in  my active time at the Faculty of
Mathematics of the University of Vienna (in the framework of the NuHAG
Group, the Numerical Harmonic Analysis Groups, see {\tt www.nuhag.eu}).} 
%
This note is certainly in the spirit of the late Abdul Siddiqi who
always tried to bring together mathematicians and engineers.


\newpage

\section{Introduction}

Basically this note is a continuation of the previous paper \cite{fe20-2},
which was a contribution to the same conference series.  While the
earlier note described the key {\it ingredients}  required for
a modern introduction to Applied Fourier Analysis and provided
some reading list, we will touch on some more practical issues,
indicating the motivation behind this enterprise. We are sure
that this is {\it in the spirit of the late Abdul Siddiqi}. Mathematicians
have to try to adapt their methods (if possible) in such a way that
engineers can make use of them without being lost in too abstract
mathematical arguments. The goal of bridging the gap between abstract mathematical concepts and real-world applications aligns well with the interdisciplinary nature of mathematics and its connections to physics.

There are many obstacles to direct communication between the two
groups. Sometimes it is just the use of a different language.
Typically   a term arising on one side is identified
with a different term appearing in the other community. Such  a problem
is relatively easy to overcome, once the bridge has been built through
conversations between members of the two groups, or by studying the publications of the other side. Then a ``vocabulary book'' can be
built and after a while newcomers do not even notice that part of the
jointly established terminology is coming from one or the other side.

For example, in wavelet theory one can clearly trace the term
{\it filter bank} to the engineering side, the idea of {\it multi-resolution
analysis} (MRA) established by S.~Mallat (\cite{ma89}, or his book \cite{ma09-1})  and Y.~Meyer
goes back to ideas in image processing, while the concept of {\it frames} is clearly a mathematical construction, which arose in the context of non-harmonic
Fourier Analysis.

There are however more serious obstacles, which have more to do with the
level of (mathematical) rigor of derivations. Mathematicians care more
- by education - for precise descriptions of terms and logical derivations
of the statements they publish, while engineers are most of the time
satisfied with heuristic derivations (seen as such by mathematicians).
It is a bold, but perhaps helpful picture to say that (pure) mathematicians
are somehow blind concerning applications, they do not ``see'' the
``physical background'' and thus they have to get precise instructions
where to move in the world. They need rules and have to distinguish between
heuristic ways of searching for a path to a new goal and the subsequent
realization. It is often like extreme climbing, where the person has to
secure his/her way step by step to be save.

In contrast, a physicist or engineers often believes to ``know'' things, because he or she has seen so many instances in
their experiment, they have taken measurements which ``clearly indicate''
how things go. They just want to understand why it works as it works,
in order to better control the situation, or improve some device, or
build some machine (like a motor, a mobile communication system, an
MRT device), in terms of efficiency or quality of output. Very often
this is actually enough to build new machines, but the more such
a device (or algorithm) is used the higher the risk is to apply it in
a situation where it does not produce the expected benefit. In the worst
case it produces an accident or unforseen output. Thus the applied
scientists may trust their intuition too much and miss dangerous
constellations.

Looking into the literature one finds that many mathematicians are
satisfied with the common view-point that a certain area (like Fourier Analysis, splines, or wavelet theory) has shown to be useful in many applications. They take this as a justification to study  questions
which appear to be somehow connected to the original problem in great detail,
loosing sometimes more and more the contact to the original problem area,
or even to developments in the applied field.
Of course it is important to understand, let me say {\it Gabor Analysis}
by getting acquainted with the corresponding function spaces
(which turn out to be the {\it modulation spaces}). As an example, the
description of {\it pseudo-differential operators} using Time-Frequency
methods has resulted in some stronger and improved versions
of classical results (Calderon-Vaillancourt paper by Heil/Groechenig \cite{grhe99}).
Meanwhile there are books covering this field (\cite{coro20},\cite{beok20}).
But as time goes on and theories get more specialized the connection
to the original setting is completely lost, and mathematical theory
might not be useful for applied scientists anymore. Even worse, it might
be potentially useful, but they cannot grasp the relevance, for
whatever reasons.  This is why communication and adequate presentation
is so important, and this author hopes to contribute a little bit
 by this article.

There is another perspective that one can take concerning this situation.
Let us view the applied scientists as those who want to realize
certain tasks, while mathematicians are providing the tools which
allow them to be successful. In this picture we have the situations
where inventive engineers have devised methods to solve a pending
problem, and then mathematicians came to refine the method. They
may justify what has been done in the applications, but more importantly
they may {\it describe the exact assumptions} that have to be made
in order to turn heuristic considerations into valid principles.
In this way they may {\it set up warning flags}  for an uncontrolled
use of such methods in areas where they cannot be validated.
While engineers are - at least in a first round - satisfied with
the observation that e.g.\ an algorithm works fine in all the
test cases, mathematicians search for arguments that ensure that
{\it ``under certain conditions the algorithm will converge with
an a priori predictable rate''} and then deliver the correct result.

The introduction of the {\it Dirac measure} (which plays an important
role in both physics and engineering) is one such example, where
the applied side was going ahead, and mathematics was following.
In the last 70 years the theory of {\it tempered distributions}
as developed by Laurent Schwartz (see \cite{sc57})
has become a backbone of mathematical
analysis. Later on L.~H\"ormander has laid the foundations for the
treatment of partial differential operators (PDEs) via this theory,
which is a cornerstone of modern mathematics (\cite{ho76}). Nowadays the
{\it sifting property}  of the Dirac Delta distribution is an
standard tool in early engineering courses on system theory , and
{\it Dirac combs} are most naturally used in order to describe the
Shannon Sampling Theorem (cf. \cite{fe24-1}).

It is obvious that in the situation described above  engineers may fail to
test potentially critical cases. Some of the disasters based on technical
faults are happening because a system may have reached a status which
was not considered during the testing phase. If this happens
to an airplane (Niki Lauda Flight 004
\begin{verbatim} https://en.wikipedia.org/wiki/Lauda_Air_Flight_004
\end{verbatim} is an example) the result
is of course causing big humanitarian loss.

On the other side of the scale mathematicians may develop more
and more refined method for problems which {\it look like applied
problems} but which unfortunately do not have any concrete impact
on the corresponding application area. The idea of a continuous
Fourier transform, defined by means of the Lebesgue integral,
on the space of (quadratically) integrable functions on the
torus or on the Euclidean space $\Rdst$ are of course beautiful
mathematical objects, and corresponding strong results can be
proved. On the other hand, the practical treatment of signals
allows only to take finitely many sampling values, which constitute
- in the spaces $\Lisp$ or $\Ltsp$ just a set of measure zero,
so are viewed as (mathematically completely) irrelevant.

Certainly the proof of the almost everywhere convergence
of Fourier series arising from square integrable functions
$f \in \LtT$ by L.~Carleson is a deep and famous result,
but does not have an impact for engineering applications,
and even very few mathematicians have studied the complicated
proof. In contrast, the Shannon Sampling Theorem or the
description of time-invariant linear systems via {\it impulse
response} or via the corresponding {\it transfer function}
are important and valuable ingredients for a good engineering
education.

Based on such considerations and views which are the result
of long-term engagement in the study of engineering applications
and problems and their mathematical treatment the author of this
notes wants to propose a new view on the concept of {\it signals}
via the theory of {\it mild distributions}.

As the term suggests mild distributions are a subclass of the
collection of {\it tempered distributions} and thus of the theory
of distributions, but unlike the setting used there the space
of test functions (now {\it Feichtinger's algebra}, the Segal
algebra $\SORdN$) is not a space of infinitely differentiable
functions, but just a (well-defined) space of continuous and
(Riemann) integrable functions, which is well adapted to the
Fourier transform. In fact, the usual (integral version of the)
Fourier transform, given by
$$  \hatf(s) = \intRd  f(t) e^{- 2 \pi i s t} dt $$
leaves $\SORd$ invariant, and consequently
the Fourier inversion, usually described as
$$  f(t) =  \intRd  \hatf(s) e^{2 \pi i t s} ds $$
is valid in a pointwise sense.

\section{What are Signals}

Asking ChatGPT provides the following answer (which is of course based on
widely used formulations in the literature, and is thus representative
for typical formulations used): \newline 
In the context of mathematics and physics, a signal generally refers to a function that conveys information. Signals can be represented mathematically and are often used in various branches of science and engineering, including physics and telecommunications. In physics, signals can represent various physical quantities, such as voltage, temperature, or sound waves, and they play a crucial role in the analysis and understanding of systems. In the field of mathematics, signal processing involves the study of mathematical techniques for analyzing and manipulating signals.

In engineering, the term ``signal'' typically refers to a time-varying physical quantity that conveys information. Signals can take various forms, such as electrical voltages, currents, or electromagnetic waves, and they are crucial in communication systems, control systems, and signal processing. Here are a few key points related to signals in engineering:  Signals can be represented mathematically as functions of time, space, or another independent variable. For example, an electrical voltage signal could be expressed as a function $V(t)$, where t is time.
 There are various types of signals:
        Analog signals or continuous signals varying smoothly over time, while
        digital  or  discrete signals take on distinct values at specific points in time or space.
  Engineers use signal processing techniques to analyze, manipulate, and extract information from signals. This includes filtering, modulation, demodulation, and various transformations.
  In telecommunications, signals are used to transmit information over long distances. Modulation is often employed to encode information onto carrier signals for efficient transmission.

\section{How Mathematicians see Signals}

In essence, mathematicians see signals as a mathematical object, and the study of signals involves applying mathematical tools and concepts to understand their properties, transformations, and behaviors. The mathematical perspective provides a rigorous framework for the analysis and synthesis of signals in various applications. As usual, mathematicians do not care about the
question of ``what signals are'', but rather develop a framework which
allows to derive certain results (theorems, based on proofs, equations,
estimates, and so on, which are valid under well-defined assumptions).
It has to be left to those who make use of such mathematical results
to decide, whether it is helpful for them or not. Sometimes it takes a while,
until the mathematical model is well enough understood in order to be actually
used. In some other cases a more comprehensive or a more simple model replaces
an older one, because it allows to obtain better results or find more elegant
ways of explaining well established statements. The theory of mild distributions is such a theory, as it can be seen as a reduction from the
wide area of tempered distributions, but on the other hand it provides
a unified approach to many questions of Fourier and Time-Frequency Analysis.
Above all, there are meanwhile simplified ways to establish the foundations
of this area (see e.g.\ \cite{feja20}, \cite{fe20-1}).


The Dirac Delta function, often denoted as $\delta(t)$, is a mathematical distribution rather than a conventional function. It is named after the physicist Paul Dirac, who introduced it as a tool in quantum mechanics.

It is a widespread habit to describe the Dirac Delta as follows:
The Dirac Delta is defined in such a way that it is zero for all values of its argument except at a single point, and the integral of the function over its entire domain is equal to 1. Mathematically, the Dirac Delta is often represented as:
$ \delta(0) = \infty$  and zero elsewhere,
with the property that \(\int_{-\infty}^{\infty} \delta(t) \, dt = 1\).
{\bf The Dirac Delta function is not a function in the traditional sense because it does not have a well-defined value at  $t = 0$, and it cannot be defined at a single point in the usual way}. Instead, it is defined through its properties, particularly its behavior under integration.
Despite not being a function, the Dirac Delta is a valuable mathematical concept, frequently used in physics and engineering to model idealized impulses or concentrations of mass, charge, or other quantities. It plays a crucial role in signal processing, Fourier analysis, and various branches of physics.

In other words, with some ``Hocus-Pocus'' mathematicians (or engineers,
accepting such manipulations) are able to make use of objects which a kind
of spooky, contradict usual mathematical rules, and are in direct
contradiction to the very well established theory of Lebesgue
integration, where one is allowed to have functions taking the
value $+\infty$, but still has the integral zero, because a single
point is a set of measure zero. So even the ``best available''
integration theory cannot be used in order to explain the properties
of Dirac Delta, and after all, why is the sum of two Dirac Deltas
(or a scalar multiple, say $7 \delta(t)$) then not the same as
$\delta(t)$, by the rule that $7 \cdot \infty = \infty$!?


The discrepancy in the view on what {\it scientific progress} is in the
different fields has led - especially in old and traditional fields such
as Fourier Analysis - to the situation, that the two communities are
only keeping very loose contacts and the majority of their
representatives is not aware of the questions discussed by the
 other side.

Therefor  it is the purpose
of this manuscript to bring the mathematical world closer
to the applied scientist (engineers, physicists, image
processing people, optical or acoustical scientists), who
need to understand the basic principles of Fourier Analysis.
We make a strong attempt to provide a lot of motivation
for the proposed approach via ``mild distributions''.

This manuscript is also a brick stone for a comprehensive
theory called ``Conceptual Harmonic Analysis'' ({\bf CHA})
see \cite{fe16}, an approach
to Harmonic Analysis which includes in some sense the
classical variants (unifying the different settings,
of continuous and discrete, periodic and non-periodic signals)
in the form of Abstract Harmonic Analysis ({\bf AHA}),
but also numerical issues and all kinds of questions arising
from time-frequency analysis. We will provide a number
of arguments, why - at least from this point of view -
the proposed setting is very {\it natural}.


\section{What are signals: Mild Distributions}

Even if it may be hard to precisely define what a signal is
on a theoretical respectively abstract basis it is much easier
to describe what are good tools to perform signal processing
tasks. For mathematicians the question may be more to ask
what are the relevant mathematical tools which can be used
to build such practical tools, and what are the concept that enable
applied scientists to make use of the underlying concepts.

There are different ways how the mathematical community can
contribute. On the one hand mathematicians can develop new
algorithms which help to solve pending problems more efficiently
(like the phase-retrieval problem, which has become a serious
mathematical subject area during the last decade), but on the
other hand it may be useful to derive new concepts. Based on
such concepts it may be able to prove that certain things
might be impossible (like a Gaborian orthonormal basis
for $\LtRN$, due to the Balian-Low Theorem) or provide
constructive methods which ensure certain approved conclusions,
respectively clarify the range of validity of a formula
(such as the Poisson Formula). Such considerations may also
lead to with granted convergence rates in
various functions spaces (such as weighted $\Lpsp$-spaces)
for the irregular sampling problem.

Also, in some case it is enough for the user to know that
she/he can rely on the expected performance of an algorithm,
{\it if well-defined} assumptions are satisfied, while in
other cases it might be good to understand the underlying
reasoning in order to modify or adapt existing tools to a
new setup.

Whatever the case is, it is always good to try to make the
concepts plausible and easy to use, if this is possible.
In addition, a tool which has a widespread range of
applicability is more likely to be used by many scientists,
as opposed to a very technical construction which is designed
only to settle a very specific problem.

In this sense the we believe that the concept of
{\it mild distributions} is a simple and widely applicable
concept which should be promoted further and which is
simple enough to be used in engineering or physics courses,
and of course in basic courses on Fourier Analysis with
an application oriented touch.

\subsection{The Heuristic View-point}

Let us return to the question of what a signal is. There is a
wide range of signals and signal processing tasks. For our illustration
let us choose an example from acoustic, where signals called
``music'' are produced, recorded, transmitted, e.g. via streaming
services and so on (see \cite{do01}, \cite{do02},\cite{bado17}).
This is something that it done (consciously
or unconsciously) by almost everybody in our modern world, using
a mobile phone or a radio. As a second example we will talk about
images, which are transmitted in a similar way, shared by friends
or for documentation of events.

So let us start with a piece of music, which we could buy -
meanwhile in an old-fashioned way - by obtaining a CD with this
3-minute piece of music stored in high quality. But what is
stored, and how is the information stored in a way that allows
the user to replay the song using suitable technical devices?
Is it really an (equivalence class of) measurable and square
integrable functions, i.e.\ a function $f(t)$ (to use physics
notation) of time which is well defined up to a null-set?
As a matter of fact we obtain (up to quantization) 44100 samples
of the sound recorded with the help of good microphones (actually
recording even frequencies beyond 20kHz).

Classical Fourier Analysis has been for almost 200 years a challenge
to the mathematical community and a number of deep results
stimulating the development of linear functional analysis,
but also Abstract Harmonic Analysis (AHA) or the Gelfand
Theory for commutative Banach algebras. Only about 100 years
after J.B.~Fourier's proposition the theory of Lebesgue integration
emerged which until now is viewed as the crucial building block
of AHA. One has to study the Banach convolution algebra $\LiGN$
(with respect to the Haar measure) in order to be able to analyze
the problem of spectral synthesis (see \cite{re68}). But does this
rich mathematical world help us to understand the effectiveness
of recording.

It is really necessary to leave the idea of a ``global Fourier
transform'' which - at least in principle - could be computed for
a whole symphony, but it would be require the best computers,
even nowadays, and at the end it would not be useful, due to the
bad localization of the building blocks, the pure frequencies.
The {\it Short-Time Fourier Transform} (STFT) appears to be the correct
tool, and leads to the theory of {\it Gabor Analysis}, i.e. a
(double) series representation of ``signals'' in terms of
time-frequency shifted copies of a {\it Gabor atom}, the so-called
Gabor representation of functions. This representation can be seen
to be at the basis of the modern MP3 compression method (see
\cite{ba05-1}).

One can say that the Segal algebra $\SORdN$ (called
{\it Feichtinger's algebra}, introduced in \cite{fe82-1}
and studied in \cite{ja18}) is an important corner stone
for the discussion of problems arising in connection
with Gabor Analysis (see e.g. \cite{fezi98}, \cite{feko98} in
\cite{fest98}, see also \cite{fest03}). The standard book
on the subject is \cite{gr01}. One can say that the dual
space, meanwhile called the space of {\it mild distributions}
(see \cite{fe19,fe20-1})
has been already used in various places, e.g. in \cite{fegr92-1}
(Theorem 5.4) in order to derive the so-called {\it kernel theorem}.

Together with the Hilbert space $\LtRdN$ the two spaces
form the so-called  Banach Gelfand Triple $\SOGTrRd$, see
\cite{cofelu08}, \cite{fe09}, \cite{fe18-3}, \cite{fe24-1}, \cite{fe24-1},
just to indicate that there is meanwhile a rich literature on the
subject, demonstrating the various areas of application outside
of {\it time-frequency analysis} (TFA). We will not pursue the
advantages of this particular setting but rather restrict
our attention to the proper use of the outer layer of this
Banach Gelfand triple, the space $\SOPRd$, which is endowed
naturally not only with the usual norm topology of a dual
space, but also with the $\wstd$topology. There are abstract
(functional analytic) arguments which imply that this space
is metrizable, because $\SORdN$ is a separable Banach space,
and hence convergence in the $\wstd$sense can be characterized
with the help of sequences (instead of the more complicated
use of so-called {\it nets}).

However, in order to avoid heavy functional analytic tools
in our discussion here we will pursue the essential features
based on the much easier (but equivalent) sequential approach
to mild distributions, as outlined in \cite{fe20-1}.

\section{The Mathematical Formulation}

In this subsection we summarize a couple of basic definitions.
We start with an elementary (alternative) characterization
of {\it Feichtinger's algebra} (\cite{rest00},\cite{ja18}), starting from scratch.

First we define the short-time Fourier transform (STFT) of
a continuous, bounded function $h \in \CbRd$ with respect
to a compactly supported {\it window function} $g \in \CcRd$.
Typically $g$ will be real-valued and non-negative, and symmetric
around zero, but one could choose any Schwartz functions, and
often it is recommended to use a Gauss function, such as the
density of the normal distribution, i.e. $g_0(t) = exp(-\pi |t|^2)$,
$t \in \Rdst$.
\begin{definition} \label{STFTdef12}
The Short-time Fourier transform of a signal $h$ with respect
to the window $g$ is given by:
\begin{equation}
V_gh (t,s) = \intRd h(y) g(y-t) e^{-2 \pi i s y} dy  = \FT( h \cdot T_t g)(s).
\end{equation}
 \end{definition}
Usually the definition is given for $g,h \in \LtRd$, and then one
can describe $V_gh$ using TF-shifts $\pi(\lambda) = M_s T_t$,
for $\lambda = (t,s) \in \TFd$,  as a scalar product in the Hilbert space $\LtRd$ (see \cite{gr01}):
$$  V_gh(t,s) = \langle h, M_s T_t g \rangle_\Ltsp =
    \langle h, \pilamg \rangle_\Ltsp. $$
For $d=1$ is customary to display $|V_g(h)|$ as a function in the (complex)
plane, with time in the horizontal and frequency in the vertical direction.

Whenever $h$ and $g$ have compact support  $V_gh$ will be supported on a strip,
say $ [-K,K]^d \times \Rdst \subset \TFd$, for some $K > 0$.
If furthermore both  $g,h$ belong to some Sobolev algebra $\HsRdN \hkr \CORdN$,  with $s > d/2$, then the set
of all products $ h \cdot T_t g$ (we can restrict our attention to
a set of parameters $t$ in some compact set, the rest is zero) is
uniformly bounded in $\HsRdN$, hence in the Fourier algebra
$\FLiRdN$\footnote{We write $\FLiRd$ for the range of $\LiRd$
under the Fourier transform, i.e.\ for $\{ \hatf \suth f \in \LiRd\}$. },
due to the Sobolev embedding theorem (and the fact that such a
Sobolev space is a pointwise algebra). Hence for any such pair
of functions $V_gh$ shows also good decay in the frequency direction
 and is definitely integrable
over phase space $\TFd$, i.e. (for fixed, non-zero $g$) on has
$$ \SOnorm{h} = \linorm{V_gh} = \intRtd |V_g(h)(\lambda)| d\lambda < \infty.$$

Instead of the usual definition which characterizes $\SORdN$ as the
subspace of all functions $f \in \LtRd$ such that $V_gf \in \LiRtd$
for one such nice bump-function $g$ we can build the space from
the building blocks described above. We fix $g$ as above and
call $h$ an {\it admissible atom} if it is of the above form,
with $\sonorm{h} :=\linorm{V_gh} < \infty$.
 Note that one has for some $C_g > 0$
$$   \ltnorm{h} \leq \linfnorm{h} + \linorm{h} \leq C_g \linorm{V_gh}. $$
\begin{proposition} \label{SOblock12}
We define $\SORd$ as the space of all absolutely convergent sums
of admissible atoms in $\CcRd \cap \HsRd$
\begin{equation} \label{SOdef12A}
\SORd =
 \{f=\sum_{n \geq 1}h_n\suth \sum_{n \geq 1}\sonorm{h_n}<\infty\}
 \end{equation}
This space coincides with $\SORdN$ as defined usually, with
the natural quotient norm as an equivalent norm, namely
$$ \|f\|_{atomic} := \inf_{admiss. repr.}  \sum_{n \geq 1} \sonorm{h_n}. $$
The sums are absolutely convergent in $\CORdN$ and  $\LiRdN$  and
consequently one has a continuous (and dense) embedding
$$ \SORdN \hkr \LiRdN \cap \CORdN \hkr \LtRdN. $$
\end{proposition}
\begin{proof}
Starting from the usual definition (see e.g. \cite{gr01}) one knows that
$\SORdN$ is a Banach space. For fixed, non-zero $g \in \HsRd \cap \CcRd \subset \FLiRd \cap \CcRd$ the norm on admissible atoms is the usual
one. Absolutely convergent sums are thus uniformly and absolutely
convergent. Since $\SORdN$ is a Banach space the definition given above
describes clearly a subspace of $\SORdN$ (as defined usually), with
a continuous embedding.

On the other hand the various characterizations of $\SORdN$ a given
e.g. in \cite{ja18} imply that any $f \in \SORd$ (defined in the classical
way as Wiener amalgam space $\WFLiliRd$) can be decomposed into an absolutely
convergent sum of the above form, by way of the following reasoning:
Given $f \in \SORd$ also $\hatf \in \SORd$, and thus it can be
decomposed into an absolutely convergent sum in $\FLiRdN$ of the form
$$ \hatf = \sumkZd  \hatf \cdot  T_k \psi $$
for some test function $\psi$ defining a BUPU, i.e. with $\sumkZd T_k \psi(x) \equiv 1.$ For any test function $\tau \in \SORd$ satisfying
$ \tau(x) \psi(x) = \psi(x) $ one also has
$$ \hatf  = \sumkZd  (\hatf \cdot T_k \psi) \cdot T_k \tau, $$
and since $\SORdN$ is a pointwise ideal in $\FLiRdN$ this implies
that also the sum is absolutely convergent in $\SORdN$, i.e.\
$$\sumkZd \sonorm{\hatf \cdot T_x \psi}\leq C_1\sonorm{\hatf}=
C_1 \sonorm{f}. $$
Each of the functions $f_k :=  \IFT(\hatf \cdot T_k \psi)$ is band-limited
and thus belongs to any Sobolev space $\HsRd$ for $s \geq 0$.
On the other hand, by the Fourier invariance of $\SORdN$ one can
decompose each $f_k$ now
on the time-side in order to come up with the claimed decomposition,
since $$ \sumjZd   \sonorm{f_k \cdot T_j \psi} < C_2 \sonorm{f_k}. $$

\noindent
The countable family given by $h_{k,j} = f_k \cdot T_j \psi $ is thus
absolutely convergent, with
$$ \sum_{k,j} \sonorm{h_{j,k}} \leq C_3 \sum_{k \in \Zdst} \sonorm{f_k}
\leq C_4 \sonorm{f}.  $$
Thus the key arguments needed to verify Proposition \ref{SOblock12} have
been given. The rest is left to the reader.
\end{proof}

\begin{remark}
An obvious consequence of the above proposition is the fact that
the compactly supported functions in $\HsRdN$ (for any $s \geq 0$)
form a dense subspace of $\SORdN$.
\end{remark}

\subsection{Why Should We Talk about Mild Distributions}

Sometimes the author has been confronted with the argument
that the elements of the dual space of general function
algebras should not be called {\it distributions} or
{\it generalized functions} as this terminology appears
to be reserved to members of the dual space of spaces of
infinitely differentiable functions on $\Rdst$.

Given the aim of having a tool which is not restricted
to the Euclidean spaces $\Rdst$ (respectively Lie groups,
where differentiability is meaningful), but also to general
locally compact Abelian (LCA) groups $\cG$ we view it as
important to reduce the request on the {\it space of test
functions}. It should be an {\it algebra}  $\Asp$ of continuous functions $f$
because then the members of the dual space can be localized,
by the simple rule $\sigma \cdot h(f) = \sigma(h \cdot f)$,
for $f \in \Asp, h \in \AspP$.
Assuming furthermore that $\Asp$ contains
functions with arbitrary small support one can define the
support $\supp(\sigma)$ of such a distribution $h$. Finally the
Fourier transform of $\sigma$ can be defined by the (simple,
usual) rule $\hatsi(f) = \sigma(\hatf), \, f \in \Asp$, if
$\Asp$ is Fourier invariant.

All this makes sense in the spirit of {\it generalized functions}
because if $\Asp \subset \LiRd \cap \CORd$, i.e.\ if it is a
continuous and absolutely (Riemann) integrable function (see
\cite{feja20})  on $\Rdst$ then any continuous and bounded function $h \in \CbRd$ defines a linear functional
$$ \sigma_h(f) := \intRd  h(t) f(t)dt. $$

Another aspect is the choice of the topology in the space
$\Asp$ of test functions. For the case $\Asp = \ScRd$, the
{\it Schwartz space of rapidly decreasing functions} convergence
in $\ScRd$ can be described only with the help of a (countable)
family of (semi)norms. Such a family allows to turn $\ScRd$
into a metric space, which is in fact complete. As a matter of
fact one can state:  $\ScRd$ is a {\it nuclear Frechet space}.
Convergence in the dual space is even more tricky to be described
correctly.

 However, if  $\AspN$ 
 is a simple (separable) Banach space  then
one has the natural $\wstd$convergence of (bounded) sequences
as a very natural tool. For the case $\AspN = \SORdN$ we
 will call this form of convergence {\it mild convergence}. It is a trivial consequence of the estimate
\begin{equation}\label{SOPconvwst}
  |\sigma_n(f) - \sigma_0(f)| \leq \SOPnorm{\sigma_n - \sigma_0} \SOnorm{f}
  \to 0, \quad \mbox{for} \,\, n \to \infty,
\end{equation}
that norm convergence of $\siginf$ implies mild convergence. Obviously
the converse is often not valid. Thus we have $T_x f \to 0$ mildly,
for $x \to \infty$ for any $f \in \SOPRd$ (even for $\siSOP$), but $\SOPnorm{T_x f} = \SOPnorm{f} \neq 0$ for any non-zero $f$.
It is also clear that $\delta_x \to \delta_y$ as $|x-y| \to 0$
mildly, but one has $ \SOPnorm{\delta_x - \delta_y} = 2 $ for $x \neq y$
(see \cite{fe17} or \cite{fe22} for precise statements).

\section{The Advantages of Mild Distributions}

It has been explained in earlier papers at length why
the space $\SOPRdN$ of mild distributions is highly useful
and provides a suitable tool for many mathematical problems
arising in applications. In fact, it can been seen as a
good substitute for many specific mathematical objects
arising in different context. It is both the simplicity
and the good properties that makes it a perfect tool
for many areas where Fourier Analysis is playing a role,
from classical Fourier Analysis to systems theory, or
from mobile communication to pseudo-differential operators.
Among the Banach spaces of mild distributions  $\MpqRdN$
called {\it modulation spaces}, with $1 \leq p,q \leq \infty$,
with $ \SORd = \MiiRd$ and $\SOPRd = \MininRd$ as smallest
respectively largest member of the family play of course a
special role.

The space $\SOPRdN$ contains not only all the function
spaces $\LpRd$ (or even Wiener amalgam spaces of the
form $\WLplqRd$), but also all the periodic functions
(with any form of periodicity), or discrete measures
supported on an arbitrary discrete lattice. Recently
it has been recognized of being very useful in the context
of {\it quasi-crystals}, replacing the theory of transformable
measures as developed in \citeX{argi90}.

\section{The Spirit of Banach Gelfand Triples}

We do not go into the details of the Banach Gelfand Triple $\SOGTrRd$,
but would like to describe a few heuristic arguments how to make use
of this tool. The comparison with the chain $\QRC$ of rational,
real and complex numbers is often helpful.

At the inner layer (the Feichtinger Algebra $\SORdN$) one has
continuous and integrable functions, which can be sampled, and also
their Fourier transforms (also in $\SORd$)
have the same property, thus Fourier Inversion is valid in a pointwise
sense (this is actually the basis of classical summability theory).
In addition, one is not far from the finite dimensional case, because
given $f \in \SORd$ it is possible to recover a good approximation
of $\hatf$ via a suitable application of the FFT algorithm.

The completion of $\SORd$ with respect to the usual scalar product
gives a valid copy of the Hilbert space $\LtRdN$, which can be embedded
into the dual space $\SOPRdN$ in the usual way. For example, the
Fourier Transform is then a unitary automorphism of $\LtRdN$ (Plancherel's
Theorem). But is also extends in a unique $\wwst$-continuous fashion
to all of $\SOPRd$, using the simple definition $\hatsi(f) = \sigma(f),
f \in \SORd$. Since $\ScRd$ is dense in $\SORdN$ this definition
(which does not require to use the Schwartz theory) can be seen as
a restriction of the generalized Fourier transform for tempered
distributions to $\SOPRd \hkr \ScPRd$ (but we avoid this description here).

As pointed out in \cite{fe20-1} one can describe the space of mild
distributions as a kind of (sequential) completion of the space
$\CbRd$ of bounded, continuous and complex-valued functions, or even
as completion of $\SORd$ itself. Either
one makes use of the analogy with the construction of real numbers
from the rational numbers (by defining real numbers as equivalence
classes of Cauchy sequences of rational numbers), or one recalls that
one can generate such an approximating sequence by regularization
(e.g. by taking partial sums of Gabor expansions, or by convolving
and localizing a given mild distributions).

Recalling that a mild distributions is a (tempered) distribution having
a bounded STFT we can easily describe the abstract concept of $\wstd$convergence by looking at the STFTs of the approximating
sequences. Instead of norm convergence in $\SOPRdN$ (which corresponds
to uniform convergence at the level of the STFT, for any fixed window)
one just has locally uniform convergence. In other words, we have:
\begin{lemma} \label{wstSOPconv1}
Let $g \neq 0$ be any window function in $\SORd$. Then
a bounded sequence $(\sigma_n)_{n \geq 1}$ in $\SOPRdN$
converges to $\sigma_0$ if and if only if the following is true:\newline
Given $R > 0$ and $\espo$ there exists $n_0$ such that one has
for $n \geq n_0$:
$$ |V_g(\sigma_n)(\lambda)-V_g(\sigma_0)(\lambda)| = |V_g(\sigma_n- \sigma_0) (\lambda)| \leq \veps, \quad |\lambda| \geq R. $$
\end{lemma}
From now on we will use the alternative term {\it mild convergence}
instead of $\wstd$convegence for bounded sequences.

When it comes to all the standard operations (translation, modulation,
dilation, taking the Fourier transform, and so on) it is a good idea
to recall that all these operations are in fact $\wwst$-continuous,
i.e.\ map bounded, $\wstd$convergent sequences as above into
equally $\wstd$convergent bounded sequences. Since $\SORd$ is
$\wstd$dense in $\SOPRd$ (hence also $\LtRd$) this implies that
the extended transform is uniquely determined.
For details see \cite{cofelu08}
or \cite{fe09}.

\section{Convergence of Mild Distributions seen Practically}

There is a very natural concept of convergence in dual Banach spaces, such
as $\SOPRd$. From the beginning there is a {\it bilinear mapping}
$$(f,\sigma): \SORd \times \SOPRd \mapsto \sigma(f)$$
with the obvious (by definition) estimate
$$ |\sigma(f)| \leq \SOPnorm{\sigma} \SOnorm{f}, \quad \fSO,\siSOP. $$
In fact, for a given bounded linear function $\sigma$ on $\SORdN$
the norm $\SOPnorm{\sigma}$ is the best constant $C \geq 0$ for estimates
of the form $ |\sigma(f)| \leq C \SOnorm{f}, \fSO.$.

Therefore it is natural to call a sequence \tred{\it mildly convergent}
(technically: $\wstd$convergent) if one has
$$ \lim_{n \to \infty}  \sigma_n(f) = \sigma_0 (f), \quad \fSORd. $$
It turns out that it better to restrict the attention to sequences
$\siginf$  which are bounded in $\SOPRdN$ (i.e. with a common
upper constant), because then the uniform boundedness principle
can be invoked in order to derive that it is just enough that
$\sigma_n(f)$ is a Cauchy sequence for every $f$ from a dense
subspace of $\SORdN$. In fact, it then follows by approximation
that it is convergent in $\Cst$ for {\it every}  $f \in \SORd$
and consequently the resulting limit defines a bounded
linear functional $\sigma_0$ (with the same estimate).
Anyway, we write $\signo$ in such a situation, and call
$\sigma_0$ the mild limit of the sequence $\siginf$. 

Recalling that $\SORd$ can be viewed as a subspace of $\SOPRd$
(the nice signals inside the space of mild distributions) it is
perhaps interesting to observe that  (using some functional analysis
as well)  such a (bounded) sequence is $\wstd$convergent
if and only if one has for
any projection $\PV$ onto a finite dimensional subspace
$\Vsp \subset \SORd$ (which is then automatically closed even in $\SOPRd$
 and allows in fact such a projection, by the Hahn-Banach Theorem)
\begin{equation}\label{projconvsign}
   \lim_{n \to \infty} \PV(\sigma_n) =  \PV(\sigma_0).
\end{equation}

Starting from this information it is  easy to verify that this type of
convergence has a very natural meaning in the context of the STFT.
One has $\wstd$convergence of $\siginf$ if and only if one has
uniform convergence of $V_g(\sigma_n)(\lambda)$ over compact
subsets of $\TFd$, or equivalently: For any Gabor family
generated by a pair $(g,\Lambda)$ with $g \in \SORd$ and
 some lattice $\Lambda$ (of the form $\Lambda = \Asp(\Zdst)$,
 for some non-singular $2d \times 2d$-matrix $\Asp$)
one has (unconditional) convergence of the finite Gabor sums
to $f$ in $\SOPRdN$. Spelled out in detail be can formulate
it as a proposition characterizing mild convergence: 

\begin{proposition} \label{wstconvsign3}
Let $\siginf$ be a bounded sequence in $\SOPRdN$. \newline
Recalling that for a pair $(g,\Lambda)$,
with $g \in \SORd$,  which generates a Gabor frame in $\LtRd$ also the
dual Gabor atom $\gd$ belongs to $\SORd$ we have: $\signo$ if and only
if one the following is true:
Given any finite  subset $F \subseteq \Lambda$ and $\espo$ one can find
some $n_0 \in \Nst$ such that one has for $n \geq n_0$:
\begin{equation}\label{GabSerconv2}
\SOnorm{\sum_{\lambda \in F}(\sigma_n-\sigma_0)(\pilamgd) \pilamg}\leq \veps.
\end{equation}
\end{proposition}
In fact, since we are talking about a finite-dimensional closed subspace
of $\SORdN$ (but also of $\LtRd$ or $\SOPRd$) it is clear that norm convergence
in $\SOsp$ is equivalent to the convergence of coordinates and thus in any
of the other norms. We have chosen the strongest possible norm to have the
emphasize the most impressive conclusion for $\wstd$convergent sequences.

\begin{remark}\label{conferrecord}
Viewed from a more practical point of view one can say: a bounded
sequence $\siginf$ of mild distributions representing the information
stored in a digital recording of a piece of music (or of an image scene)
is close to the ``true underlying piece of music'' if it contains all the
required information to reproduce the sound of the recorded piece of music
for the duration of the song recorded, and up to a certain frequency.

Isn't it true, that we are completely satisfied with the ``exact reproduction''
of a piece of music when replayed on by a good device using the information
stored on a high quality CD, which contains the sampled version (and thus
all the required information for exact recovery up to $20$ kHz) of the musical
signal, for the duration of the song.

In this sense a digitally stored version of a song is nothing but a
$\wstd$approximation of the true signal (and a bat would complain that
the high frequencies are really missing!) using information from a finite
dimensional subspace of mild distributions. Similar comments could be given
for a high-resolution pixel image taken with a modern digital camera. \end{remark}

It is also possible to derive from this result that a bounded set
$M \subset \SORdN$ is {\it relatively compact} if and only if for every
$\epso$ one can find (in the situation of Proposition \ref{wstconvsign3})
a finite subset $F_0 \subset \Lambda$  such that one has:
 for any finite (or even infinite) set $F \supseteq F_0$ in $\Lambda$ one has:
\begin{equation}\label{TFtight3}
  \sonorm{f - \sum_{\lambda \in F}  \langle f, \pilamgd \rangle \pilamg}
  \leq \veps, \quad f \in M.
\end{equation}

\section{Extending Operators to Mild Distributions}

In this section we want to indicate how the elementary approach to
mild distributions (e.g. by the sequential approach, or using TF-analysis,
or more sophisticated functional-analytic methods) can be used to show
how many important operators can be extended naturally to mild distributions.

Recall that the Segal algebra $\SORdN$ has a large number of good properties.
Not only is it isometrically invariant under time-frequency shifts (and is
essentially the smallest Banach space with this property), it is also Fourier invariant (again isometrically, if one uses the STFT with a Gaussian window
for the norm), but also invariant under automorphisms (such as rotations, dilates) or even the unitary operators arising from the metaplectic group,
including Fractional Fourier transforms. It is also a Banach algebra under
convolution (in fact an Banach ideal in $\LiRdN$) and a pointwise algebra
(in fact, a Banach ideal inside of the Fourier algebra $\FLiRdN$).

These good properties allow to transfer, in fact extend in a unique $\wwst$-continuous way many of the operators to the space of mild distributions. In the terminology of the Banach Gelfand Triple $\SOGTrRd$
this means, that a unitary mapping leaving $\SORdN$ invariant can be
viewed as a Banach Gelfand Triple automorphism. We do not go into the
technical details here, but rather point to Theorem 7.3.3 in \cite{feko98}
which reads as follows:
\begin{theorem}
(Extension of Unitary Gelfand Triple Isomorphism)   \newline
A unitary mapping $U$ acting from $\Ltsp(G_1)$ to $\Ltsp(G_2)$ extends
to an isomorphism of the corresponding BGTr $\SOGTr$ if and only if
the restrictions of $U$ and $U^*$ are bounded linear operators
between $\SOsp(G_1)$ and $\SOsp(G_2)$.
\end{theorem}
As a corollary one has Corollary 7.3.4 there:
\begin{corollary} \label{isomBGTR3}
An isomorphism $V: \SOsp(G_1) \to \SOsp(G_2)$ extends to a unitary
BGT-isomorphism between $\SOGTr(G_1)$ and $\SOGTr(G_2)$ if and only if
\begin{equation}  \label{prepViso}
 \langle f_1, f_2 \rangle_{\Ltsp(G_1)} =  \langle V(f_1), V(f_2) \rangle_{\Ltsp(G_2)}, \quad f_1,f_2 \in \SOsp(G_1).
\end{equation}
\end{corollary}
This general principle can be used in order to show that
the usual definition of the Fourier transform of mild distributions,
given by
\begin{equation}\label{FTmild03}
   \hatsi(f) := \sigma(\hatf), \quad \fSORd,
\end{equation}
defines in fact a valid extension of the ordinary Fourier transform
leaving $\SORdN$ invariant. In fact, due to its $\wwst$-continuity
(i.e. the preservation of bounded, $\wstd$convergent sequences) it
is uniquely determined. In fact, the validity of (\ref{prepViso})
is in this case just the verification of the {\it Fundamental
Relationship for the Fourier Transform} (as Hans Reiter has called it),
namely the validity of
\begin{equation}\label{FundFTequ03}
   \intRd  \hatg(t) f(t) dt = \intRd  g(s) \hatf(s) ds, \quad f,g \in \SORd.
\end{equation}

For those who prefer to avoid functional analytic arguments one could
take the elementary approach, which requires some home-work, but not
really any difficult computations. It is just enough (referring
to the sequential approach to mild distributions) that the
ordinary Fourier transform (combined with the validity of the
Fourier Inversion theorem) that the (naive) Fourier transform
defined on $\SORdN$ (via Riemann integrals) maps mild Cauchy sequences
into mild Cauchy sequences, and preserves equivalence classes of
mild Cauchy sequence, the so-called ECMIs. Hence it can be extended
to an isometric mapping between ECMIs in a way which prolongates
the ordinary Fourier transform.

Similar arguments apply to other operations, and in particular
to time-frequency shifts. In fact, in the realm of tempered distributions
the space of mild distributions is the largest space, as expressed by
the following formulation:
\begin{lemma} \label{SOPmax03}
Let $\BspN$ be a Banach space of tempered distributions, continuously
embedded into $\ScPRd$, i.e. satisfying the condition that
\begin{equation} \label{BconvScP}
\lim_{n \to \infty} \Bnorm{h - h_0} = 0  \quad \Rightarrow \quad
 \lim_{n \to \infty}(h - h_0)(f) = 0, \forall f \in \ScRd.
 \end{equation}
Than the condition $$\Bnorm{\pilam(h)} = \Bnorm{M_s T_t(h)} = \Bnorm{h}, \quad (t,s) \in \TFd \,\, \mbox{for all} \,\,\, h \in \Bsp $$
implies the continuous embedding $\BspN \hkr \SOPRdN$.
\end{lemma}
\begin{proof}
It is an immediate consequence of (\ref{BconvScP}) that one can apply
the Closed Graph Theorem in order to obtain the following consequence
of the assumptions:  Every $g \in \ScRd $ defines a bounded linear
functional on $\BspN$ via the following assignment (viewing $h \in \Bsp
\subset \ScPRd$ as a tempered distribution) 
$ h \mapsto \sigma_h(\overline{g}) = h(\overline{g})$.
Fixing some non-zero $g \in \ScRd \subset \SORd$  there exists some constant
$C = C(g,\Bsp) > 0$ such that
$$ |h(g)| \leq C_g \Bnorm{h}, \quad h \in \Bsp. $$
Using the invariance property of $\BspN$ one concludes that
\begin{equation}\label{Binvar03}
  |V_{\overline{g}}h(\lambda) = |h(\pilam\overline{g})|  = | (\pilam^* h)(g)|
\end{equation}
which implies then for every $\lambda \in \TFd$ and $h \in \Bsp$:
\begin{equation}\label{Binvar03b}
  |V_{\overline{g}}h(\lambda)  \leq C_g \Bnorm{\pilam^*h} = C_g \Bnorm{h}.
\end{equation}
Thus $h$ has a bounded STFT with for the window $g \in \ScRd$, or $h \in \SOPRd$.
\end{proof}


\section{Convolution and Multiplication}

One cannot expect that convolution or multiplication extends to the space
of all mild distribution. How should one convolve the function constant $1$,
we write ${\bf 1}$ by itself? In a similar way there is no natural interpretation of the pointwise square of the Dirac Delta measure\footnote{The much more involved Colombeau Calculus discusses highly technical ways
of giving it a meaning and also demonstrates that it can be used for very
particular applications in physics, see \cite{co92-3}. But such a theory
is far beyond the scope of this approach, mostly to the limited range of
applications and the high degree of technicalities required to even formulate it. This is in spirit almost the opposite of the ideas which the present article is trying to promote: Give a simplified approach with a wide, but
limited range of applications.}.

On the other hand, it is quite natural to define the convolution of a
functional (some $\sigSOP$) with a test function $\fSORd$. In fact, there
are different methods to define the convolution. One could argue, that
$\SORd$ is a {\it homogeneous Banach space} (see \cite{fe22})  and thus one has
$\LiRd \ast \SORd \subset \SORd$ (always together with corresponding
norm estimates, namely $\SOnorm{g \ast f} \leq \linorm{g} \SOnorm{f}$),
but also, at least for continuous functions in $\Lisp \cap \COsp(\Rdst)$
it can be defined in the pointwise sense. All these different notations
are in fact compatible (this requires some proofs, but as a matter of
fact, and for the user who just wants to employ the expected, natural
identities it is enough to know it), and thus the pointwise realization
of the convolution is one particularly useful option:
\begin{definition} \label{SOPSOconv3}
Given $\sigSOP$ and $\fSO$ the {\it pointwise convolution} is
given by (writing $\fchk(x) = f(-x)$):
\begin{equation}\label{ptwsigastf3}
  \sigma \ast f(x) = \sigma(T_x \fchk), \quad x \in \Rdst.
\end{equation}
\end{definition}
The fact that $\SORd$ is isometrically translation invariant with
continuous shift (i.e. is a homogeneous Banach space) implies that
one has
\begin{equation} \label{SOPSOconv3}
 \SOPRd \ast \SORd \subset \CbRd,
\end{equation}
in fact, $\sigma \ast f$ is a uniformly continuous function on $\Rdst$,
since one has for $\fSO$
$$  |\sigma \ast f(x) - \sigma \ast f(x+h)| \leq \SOPnorm{\sigma}
 \SOnorm{T_h f -f} \to 0, \quad \mbox{for} \,\, h \to 0. $$
This fact can also be used in order to prove that $\CbRd$ is
$\wstd$dense in $\SOPRd$ (by letting $f$ in $\SORd$ run through
a Dirac family, e.g. by compressing a non-negative function
$f_0$ with $\hatf(0) =1$).

This pointwise definition can be also extended to a convolution
by bounded measures,  hence also by functions $g \in \LiRd$,
with
\begin{equation}\label{LiSOP03}
   \SOPnorm{g \ast \sigma} \leq \linorm{g} \SOPnorm{\sigma} \, ,
   \quad  g \in \LiRd, \sigma \in \SOPRd.
\end{equation}
However, the action of $\LiRd$ by convolution is not only
norm continuous on $\SOPRdN$, but it is also $\wwst$continuous,
and thus a BGTr homomorphism in the sense of \cite{cofelu08}.

Let us just formulate some basic statements concerning the preservation
of $\wstd$convergence under basic operations.
\begin{lemma} \label{convwst03}
Assume that $\siginf$ is a bounded sequence of mild distributions with
mild limit $\sigma_0$. Then one has:
\begin{enumerate}
  \item For any bounded measure $\mu \in \MbRd$ one has
   $$ \wstlim_{n \geq 1} \,\, \mu \ast \sigma_n = \mu \ast \sigma_0. $$
  \item For any function $h \in \FLiRd$ one has:
   $$ \wstlim_{n \geq 1}\,\, h \cdot \sigma_n = h \cdot \sigma_0. $$
\end{enumerate}
\end{lemma}
\begin{remark} \label{convwstrm03}
We do not give formal proof, but mention instead: For any mild
Cauchy-sequence $(f_n)_{n\geq 1}$ in $\SORd$ it is clear that $\mu \ast f_n$
is also a mild Cauchy sequence (and furthermore equivalence is preserved),
since one has
$$   \mu \ast f_n (h) = f_n (\mu\chck \ast f_n) $$
with $\mu \chck (k) = \mu(k \chck)$ for $k \in \CcRd$.

The $\wstd$convergence is also valid on the side of measures. If
$\mu_k$ is a bounded and tight (uniformly concentrated) sequence
of bounded measures which is $\wst$-convergent to $\mu_0$
then $\mu \ast \sigma$ is the $\wstd$limit of $\mu_k \ast \sigma$.
\end{remark}

\subsection{Sampling goes to Periodization}

It is a well-known fact that sampling of a function $f$ corresponds
to periodization of its Fourier transform $\hatf$.  This can be
seen as a consequence of the (discrete version) of the Poisson
formula in the finite, discrete case. In fact, it can be derived
using simple exponential sums involving unit roots of finite order
in the complex plane and the fact that they add up to zero (for $n \geq 1$).

In the setting of mild distributions an elegant proof the Shannon Sampling
Theorem for band-limited functions can be derived via mild distributions
based on the following reasoning. First we note that
$$ f \mapsto  \SHZd(f) := \sumkZd \delta_k(f) =  \sumkZd f(k)$$
(an absolutely convergent sum in $\Cst$)
is a well defined bounded linear functional on $\SORd$, i.e. a
mild distribution.

Then one observes that the {\it Poisson Summation Formula}
is not only valid for $f \in \ScRd$, but for the much larger space $\SORd$:
$$     \sumkZd \delta_k(f) =  \sumkZd f(k) = \sumnZd \hatf(n) = \sumkZd \delta_k(\hatf), \quad \fSORd. $$
This fact can be described as $ \FT(\SHZd) = \SHZd$, with
$ \SHZd = \sumkZd \delta_k$.
Since convolution products
are mapped into pointwise convolution and vice versa  we  obtain
\begin{equation} \label{pertosamp03}
\FT(\Shah \ast f)  = \FT(\Shah) \cdot \FT(f) = \Shah \cdot \hatf,
\quad \fSORd, \end{equation}
and combined with the well-known formula
$ \Shah \cdot h = \sumkZd h(k) \delta_k, \quad  k \in \Zdst $
for $h \in \SORd$  and the observation that
$ \Shah \ast f = \sumkZd \delta_k \ast f = \sumkZd T_k f$,
tells us that (\ref{pertosamp03}) tells us that periodization
of $\fSO$ corresponds to sampling on the Fourier transform side.
By applying the inverse Fourier transform this shows that
sampling on the time-side corresponds to periodization on the
Fourier transform side.

This can easily be used to recover $\hatf$ in the band-limited
case (i.e. if $\supp(\hatf)$ is compact) from its periodized
version as long as the periodization is coarse enough. This
is equivalent to recovery of a band-limited function $\fSO$
from sufficiently fine samples.

Let us illustrate this shortly for the one-dimensional
case:   coarse periodization on the
Fourier side corresponds to convolution with a dilated
Dirac comb of the form $ \Shah_\beta := \sumkZd \delta_{\beta k}$,
for sufficiently large $\beta$. In fact $\wstlim_{\beta \to \infty} = \delta_0$ and consequently $ \Shah_\beta \ast f \to \delta_0 \ast f = f$ in the
$\wstd$sense.

Applying the inverse Fourier transform the corresponding statement read
as follows: Since  $$\IFT(\Shah_\beta) = \alpha \,\, \Shah_\alpha
\quad \mbox{for}  \quad  \alpha = 1/\beta $$
one can expect that for any $g \in \SORd$ one has : 
$$  
 g \cdot (\alpha \, \Shah_\alpha)
  = \alpha \,\sumkZd g(\alpha k) \delta_{\alpha k} \to
 g \,\, \mbox{for}\,\, \alpha \to 0.$$
This fact can also be derived directly, by applying the mild distribution
$ g \cdot (\alpha \, \Shah_a)$ on a test function $f \in \SORd$, and observing
that we  simply have to consider the convergence of Riemannian sums for $ g\cdot f \in \SORd$ (which is absolutely Riemann integrable):
$$  [g \cdot (\alpha \, \Shah_a)] (f) = \alpha \sumkZd g(\alpha k) f(\alpha k)
\to  \intRd g(t) f(t) dt = g(f), \,\, \mbox{as} \,\, \alpha \to 0.$$

This estimate  could also be described by  the rule
$ [g \cdot (\sigma)] (f) = \sigma(g \cdot f)$,  valid for
any $\sigSOP$, and $f,g \in \SORd$,   combined with the
statement that $$ \wstlim_{\alpha \to 0} \alpha \, \Shah_a = {\bf 1} = \IFT(\delta_0), $$
and finally with the fact that pointwise multiplication by $g$ preserves
$\wstd$convergence.
%
The reader will find more details on this kind of reasoning in
\cite{fe24-1}.

\section{There is just {\it one}  Fourier Transform}

The setting of mild distributions also allows to make the
connection between discrete and continuous, periodic and non-periodic
signals and the corresponding formulas for Fourier transforms more
transparent. In fact, one can derive the validity (including the
concrete description) of one from any of the others. Typically
this is presented as a vague, heuristic argument, but in the
given setting this is obtained by taking \tred{``mild limits''}.
This is very much in the spirit of Jens Fischer's work, see
e.g. \cite{fi18}, or the pre-runner \cite{fe17-1}\footnote{While
the theory of tempered distributions appears to be more general, and
certainly is, with respect to questions of differentiability, the
larger reservoir of tempered distributions also has some negative
effects. Above all one cannot draw the conclusion that a
tempered distribution supported on a discrete set is a sum
of Dirac Deltas. It can well be a sum of partial derivatives, for
example. This makes some considerations more complicated or even
prohibit certain conclusions that are possible in the context of
mild distributions.}.

Since mild convergence is best understood by looking at the
Time-Frequency realizations of the corresponding signals
let us illustrate the situation in this way.

Recall that for any $f \in \SORd$ the periodized version
is of the form $ \Shah_\beta \ast f \in \CbRdN$, and hence
the sampled and periodized version (equal to the periodized
and sampled version of $f$, if $\beta = N \alpha$, so e.g.
if $\beta \in \Nst$ and $\alpha = 1/\beta)$  is
\begin{equation} \label{persmpcomm03}
(\Shah_\beta \ast f) \cdot \Shah_\alpha =
  (\Shah_\alpha \cdot f) \ast \Shah_\beta
  \end{equation}
belongs to $\SOPRd$ (as a weighted Dirac comb with bounded
coefficients) and thus has a Fourier transform in the sense
of $\SOPRd$. Following the algebraic rules of the (generalized)
Fourier transform it is not difficult to verify that one
can write both this periodic and discrete signal as a weighted
sum of finitely many shifted Dirac combs of the form
$\Shah_\beta$, and the same is true for its Fourier transform
(also now periodic with period $1/\alpha = \beta$!) and the
transformation rule (which is unitary at the level of
these coefficients in $\Cst^N$, for $N = \beta/\alpha = \beta^2$)
is just the standard DFT (Discrete Fourier Transform, typically
realized with the help of the FFT, the Fast Fourier Transform).

We thus can establish an analogy between the recovery of a
band-limited function of two variables, whose Fourier version
might look like the STFT of our signal. Recovery from the
periodized version is then possible if the periodization is
coarse enough. It can be performed by some kind of filtering,
i.e.\ multiplication on the $2$-dimensional Fourier domain.
In the current situation one can imagine well (and corresponding
theoretical investigations are on the way) that it is enough
to recover the signal (at least with good approximation) by
reconstructing the signal using the Gabor coefficients involving
the complex coefficients corresponding to the central domain.
Alternatively one may think (this is the method in the background
of \cite{feka07}) of first smoothing the discrete and periodic
function and then localizing it (or vice versa) by multiplying
it with a local plateau function. Using (again) the function
space $\SORdN$ this can in fact be achieved with arbitrary
high precision, by letting the period $\beta$ tend to infinity (and
accordingly the sampling rate $\alpha$ tends to zero).
\begin{figure}  [ht]
\label{demsmp03B2.jpg}
\begin{center}
\includegraphics[scale=0.36]{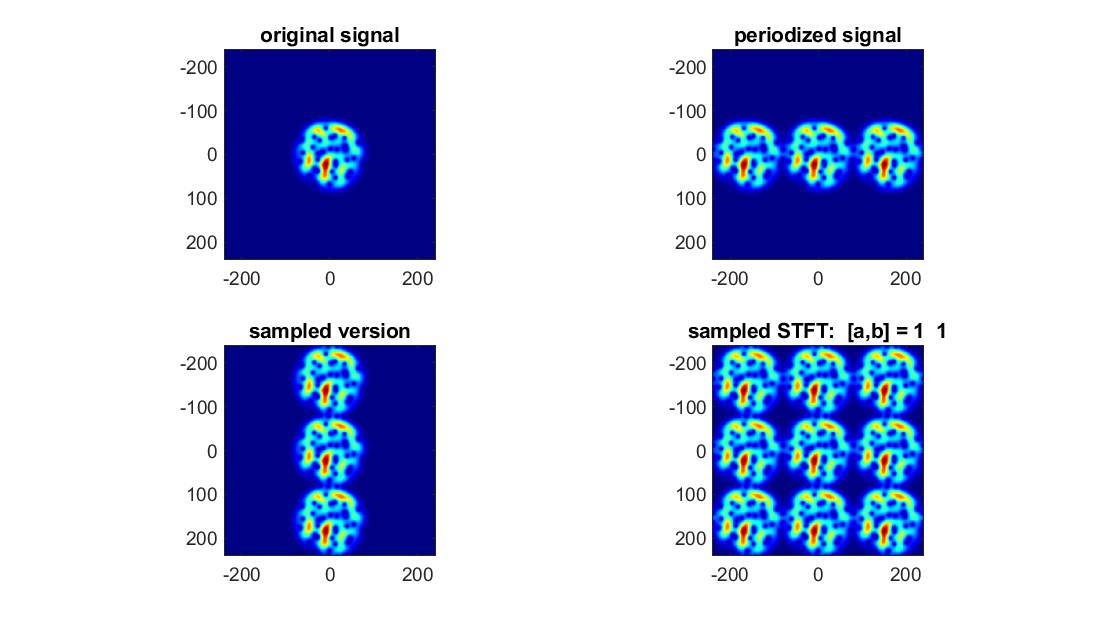}
\caption{Demonstration of the effect of periodization
and sampling in the STFT picture:  Left upper corner: original STFT of a well concentrated function; right upper corner: The Spectrogram of the a
coarsely periodized version of the same signal; left lower corner: the same
for a sampled version, i.e.\ with periodization in the frequency direction;
right lower corner: the STFT of a sampled {\it and} periodized version of the
original signal. Near the origin of the TF-plane it is exactly the same as
the STFT of the original signal.}
\end{center}
\end{figure}

Similar arguments can be made in order to describe the usual heuristic
arguments allowing the transition from periodic to non-periodic signal,
for example. Given a well-concentrated function $f \in \SORd$ (thus
with $\hatf \in \SORd$) we could look at the periodic versions of this
functions, i.e. at the mild distributions of the form
$ \fper = \Shah_\beta \ast f$. By the convolution theorem  its  Fourier transform can be computed as
\begin{equation} \label{persmpalbe3}
\FT(\fper) = \FT(\Shah_\beta) \cdot \fhat = \alpha (\Shah_\alpha \cdot \hatf) = \sumkZd  \alpha \, \hatf(\alpha k) \delta_{\alpha k}.
\end{equation}
The classical theory of Fourier Series would start from this fact, and
use the (orthogonal) expansion of $\beta-$periodic functions in terms
for the pure frequencies $\chi_{\alpha k}(t)  = exp(2 \pi i \alpha k t)$
as a starting point, and would obtain the coefficients for the expansion
for the (in fact absolutely convergent) Fourier series representation
$$  \fper(t) = \sumkZd  c_k \chi_{\alpha k}(t) $$
in the classical way, by integrating $\fper(t)$ against the
pure frequencies $\chi_{\alpha k}$ over the fundamental domain of the
periodic function $\fper \cdot \chi_{\alpha k}$. But as a matter
of fact (using the periodicity) this is just the same as the
sampling value, in other words one has  $c_k = \hatf(\alpha k)$!

Thus the heuristic transition from periodic Fourier series to
the continuous Fourier transform (first now for functions in $\SORd$,
and then extended to more general function classes) can be made
precise in the sense of mild convergence. Clearly we have
$$  \wstlim_{\beta \to 0} (\Shah_\beta \ast f) =
        [\wstlim_{\beta \to 0} (\Shah_\beta)] \ast f = \delta_0 \ast f = f
$$
or on the Fourier transform side we have
$$\hatf = \wstlim_{\alpha \to 0} \, \alpha \,\,(\Shah_\alpha \cdot \hatf).$$
Note that the correct normalization of the samples on the Fourier transform
side is by the scalar $\alpha$ as it appears in the use of Riemannian sums,
because otherwise things would tend to infinity.



\section{Generalized Stochastic Processes}

The setting described is not limited to so-called {\it deterministic
signals}, but can also be extended to the {\it stochastic setting}.
While the classical theory of {\it wide-sense stationary processes}
describes the setting by means of vector-valued measures (in analogy
with classical Fourier theory, making use of the asymmetric viewpoint
that the Fourier or inverse Fourier transform requires integrability),
the approach via mild distributions provided in the PhD thesis of
W.~H\"ormann   (\cite{ho89} and presented in   \cite{feho14})
starts by considering a {\it generalized stochastic process} (GSP)
as an abstract operator from $\rho$ from $\SORd$ to some Hilbert space $\Hilb$,
 defined via probability theory. 

In this setting one has a natural notation of a
{\it spectral process} (namely the Fourier transform of the given
process, provided via $\widehat{\rho}(f) = \rho(\hatf), \,\, f \in \SOsp$,
as expected.
The {\it spectral representation} of a stochastic process is then
nothing else than a different view on the Fourier inversion theorem.
Any such GSP $\rho$ has an {\it autocorrelation} $\sigma_\rho \in \SOPRtd$
and the autocorrelation of the spectral process $\widehat{\rho}$ is just
the two-dimensional (generalized) Fourier transform of the autocorrelation
of the original process, in short
$\sigma_{\widehat{\rho}} = \widehat{\sigma_\rho}$.
 Wide-sense stationarity corresponds to the fact
that $\sigma_\rho$ is invariant with respect to the diagonal subgroup
 $ \Delta_d:=  \{(x,x) \suth x \in \Rdst \}$ and thus equivalently to
the concentration of $\sigma_{\widehat{\rho}}$ on the orthogonal
complement $\Delta^{\perp}_d \equiv \Rdst$. In fact, in this case
it can be even identified with a translation-bounded measure.

 We do not go into further details here, but point to the fact that
 also in this setting we have a simplification of the Fourier analytic
 tasks. Furthermore the functorial properties of $\SOsp$ and its dual
 space allow to go through the same line of arguments even in the
 more general setting of LCA groups\footnote{Let us just mention
 that the word ``mild distributions'' does not appear on the referred
 papers.}

\newpage 
\section{Conclusion and Summary}

This note tries to convey the message that a new approach to Fourier
analysis is necessary in order to provide a mathematically correct, yet
simple interpretation of the key facts required for a good understanding
of the mathematics behind signal processing or system theory. Also
it should be seen as a contributions against the mystification of the
role of the Dirac distribution (a harmless bounded measure or mild
distribution).

It seems to be unavoidable to make use of basic functional analytic
tools once continuous variable come into play, simply because them
methods from linear algebra are not sufficient anymore, once the
signal spaces are too big to have a finite dimensional basis. Thus
{\it Linear Functional Analysis} has to be invoked in order to
allow infinite sequences or series, to describe the notation of
infinite series representations and so on. This is also why
Banach spaces are so important, because such spaces are complete,
very much like the real or complex numbers. Every Cauchy sequence
is convergent, any absolutely convergent series is (unconditionally)
convergent, and so on. However, in some cases even this nice concept
of convergence is not appropriate. As demonstrated in this note
the notation of ``mild convergence'' (i.e.\ the application of the
abstract principle of $\wstd$convergence in the space of mild
distributions $\SOPRd$, which is the dual space of Feichtinger's
algebra $\SORdN$) has an important role, but also a very natural
interpretation in the spirit of Time-Frequency Analysis.

Overall, the text tries to bring together a couple of ideas and
viewpoints which should help to grasp these notions and make them
more accessible to the applied scientists, be it engineers
who may design signal processing algorithms, or physicists who
tend to make use of the Dirac Delta notation often in a vague
way. The space of mild distributions should be thought as the
space of signals, and each such signals has a (generalized)
Fourier transform, which is also a mild distribution. Simple
rules apply, including the convolution theorem.

The author of this note is continuing to work on alternative
methods to simplify the approach to mild distributions or to
motivate its use, which finally should end up in new ways
of teaching Fourier Analysis to Applied Scientists but also
application oriented mathematicians. Of course he is open
to questions, constructive or even critical comments. Possible
unclear formulations are solely in his responsibility.

\newpage

\bibliographystyle{abbrv}  

\begin{thebibliography}{10}

\bibitem{ba05-1}
P.~{B}alazs.
\newblock {\em {R}egular and {I}rregular {G}abor {M}ultipliers with
  {A}pplication to {P}sychoacoustic {M}asking}.
\newblock PhD thesis, {U}niversity of {V}ienna, 2005.

\bibitem{bado17}
R.~{B}ammer and M.~{D}{\"o}rfler.
\newblock {I}nvariance and stability of {G}abor scattering for music signals.
\newblock In {\em {P}roceedings of {S}{A}{M}{P}{T}{A} 2017}, 2017.

\bibitem{beok20}
A.~{B}enyi and K.~A. {O}koudjou.
\newblock {\em {M}odulation {S}paces. {W}ith {A}pplications to
  {P}seudodifferential {O}perators and {N}onlinear {S}chr{\"o}dinger
  {E}quations}.
\newblock {A}ppl. {N}um. {H}arm. {A}nal. ({A}{N}{H}{A}). {S}pringer
  ({B}irkh{\"a}user), {N}ew {Y}ork, 2020.

\bibitem{co92-3}
J.-F. {C}olombeau.
\newblock {\em {M}ultiplication of {D}istributions}, volume 1532 of {\em
  {L}ecture {N}otes in {M}athematics}.
\newblock {S}pringer-{V}erlag, {B}erlin, 1992.

\bibitem{cofelu08}
E.~{C}ordero, H.~G. {F}eichtinger, and F.~{L}uef.
\newblock {B}anach {G}elfand triples for {G}abor analysis.
\newblock In {\em {P}seudo-differential {O}perators}, volume 1949 of {\em
  {L}ect. {N}otes {M}ath.}, pages 1--33. {S}pringer, {B}erlin, 2008.

\bibitem{coro20}
E.~{C}ordero and L.~{R}odino.
\newblock {\em {T}ime-{F}requency {A}nalysis of {O}perators and
  {A}pplications}.
\newblock {D}e {G}ruyter {S}tudies in {M}athematics, {B}erlin, 2020.

\bibitem{do01}
M.~{D}{\"o}rfler.
\newblock {T}ime-frequency {A}nalysis for {M}usic {S}ignals. {A} {M}athematical
  {A}pproach.
\newblock {\em {J}ournal of {N}ew {M}usic {R}esearch}, 30(1):3--12, 2001.

\bibitem{do02}
M.~{D}{\"o}rfler.
\newblock {\em {G}abor {A}nalysis for a {C}lass of {S}ignals called {M}usic}.
\newblock PhD thesis, 2002.

\bibitem{fe82-1}
H.~G. {F}eichtinger.
\newblock {A} compactness criterion for translation invariant {B}anach spaces
  of functions.
\newblock {\em {A}nalysis {M}athematica}, 8:165--172, 1982.

\bibitem{fe09}
H.~G. {F}eichtinger.
\newblock {B}anach {G}elfand triples for applications in physics and
  engineering.
\newblock In {\em {A}{I}{P} {C}onference 2009}, volume 1146 of {\em {A}{I}{P}
  {C}onf. {P}roc.}, pages 189--228. {A}mer. {I}nst. {P}hys., 2009.

\bibitem{fe16}
H.~G. {F}eichtinger.
\newblock {T}houghts on {N}umerical and {C}onceptual {H}armonic {A}nalysis.
\newblock In A.~{A}ldroubi, C.~{C}abrelli, S.~{J}affard, and U.~{M}olter,
  editors, {\em {N}ew {T}rends in {A}pplied {H}armonic {A}nalysis. {S}parse
  {R}epresentations, {C}ompressed {S}ensing, and {M}ultifractal {A}nalysis},
  {A}pplied and {N}umerical {H}armonic {A}nalysis, pages 301--329.
  {B}irkh{\"a}user, {C}ham, 2016.

\bibitem{fe17}
H.~G. {F}eichtinger.
\newblock {A} novel mathematical approach to the theory of translation
  invariant linear systems.
\newblock In I.~{P}esenson, Q.~{L}e {G}ia, A.~{M}ayeli, H.~{M}haskar, and
  D.~{Z}hou, editors, {\em {R}ecent {A}pplications of {H}armonic {A}nalysis to
  {F}unction {S}paces, {D}ifferential {E}quations, and {D}ata {S}cience.},
  {A}pplied and {N}umerical {H}armonic {A}nalysis., pages 483--516.
  {B}irkh{\"a}user, {C}ham, 2017.

\bibitem{fe17-1}
H.~G. {F}eichtinger.
\newblock {A}bel {P}rize 2017 for {Y}ves {M}eyer.
\newblock {\em {I}nternat. {M}ath. {N}achr.}, 236:13 -- 24, 2017.

\bibitem{fe18-3}
H.~G. {F}eichtinger.
\newblock {B}anach {G}elfand {T}riples and some {A}pplications in {H}armonic
  {A}nalysis.
\newblock In J.~{F}euto and M.~{E}ssoh, editors, {\em {P}roc. {C}onf.
  {H}armonic {A}nalysis ({A}bidjan, {M}ay 2018)}, pages 1--21, 2018.

\bibitem{fe19}
H.~G. {F}eichtinger.
\newblock {C}lassical {F}ourier {A}nalysis via mild distributions.
\newblock {\em {M}{E}{S}{A}, {N}on-linear {S}tudies}, 26(4):783--804, 2019.

\bibitem{fe20-1}
H.~G. {F}eichtinger.
\newblock {A} sequential approach to mild distributions.
\newblock {\em {A}xioms}, 9(1):1--25, 2020.

\bibitem{fe20-2}
H.~G. {F}eichtinger.
\newblock {I}ngredients for {A}pplied {F}ourier {A}nalysis.
\newblock In {\em {S}harda {C}onference {F}eb. 2018}, pages 1--22. {T}aylor and
  {F}rancis, 2020.

\bibitem{fe22}
H.~G. {F}eichtinger.
\newblock {H}omogeneous {B}anach spaces as {B}anach convolution modules over
  ${M} ({G})$.
\newblock {\em {M}athematics}, 10(3):1--22, 2022.


\bibitem{fe24-1}
H.~G. {F}eichtinger.
\newblock {THE} {B}anach {G}elfand {T}riple and its role in classical {F}ourier
  analysis and operator theory.
\newblock In H.~G. {F}eichtinger, R.~{D}uduchava, E.~{S}hargorodsky, and
  G.~{T}ephnadze, editors, {\em {T}bilisi {A}nalysis and {P}{D}{E} {S}eminar},
  pages 67--75, {C}ham, 2024. {S}pringer {N}ature {S}witzerland.

\bibitem{fegr92-1}
H.~G. {F}eichtinger and K.~{G}r{\"o}chenig.
\newblock {G}abor wavelets and the {H}eisenberg group: {G}abor expansions and
  short time {F}ourier transform from the group theoretical point of view.
\newblock In C.~K. {C}hui and C.~K. editor: {C}hui, editors, {\em {W}avelets: a
  {T}utorial in {T}heory and {A}pplications}, volume~2 of {\em {W}avelet
  {A}nal. {A}ppl.}, pages 359--397. {A}cademic {P}ress, {B}oston, 1992.

\bibitem{feho14}
H.~G. {F}eichtinger and W.~{H}{\"o}rmann.
\newblock {A} distributional approach to generalized stochastic processes on
  locally compact abelian groups.
\newblock In G.~{S}chmeisser and R.~{S}tens, editors, {\em {N}ew {P}erspectives
  on {A}pproximation and {S}ampling {T}heory. {F}estschrift in honor of {P}aul
  {B}utzer's 85th birthday}, pages 423--446. {C}ham:
  {B}irkh{\"a}user/{S}pringer, 2014.

\bibitem{feja20}
H.~G. {F}eichtinger and M.~S. {J}akobsen.
\newblock {D}istribution theory by {R}iemann integrals.
\newblock {\em {M}athematical {M}odelling, {O}ptimization, {A}nalytic and
  {N}umerical {S}olutions}, pages 33--76, 2020.

\bibitem{feka07}
H.~G. {F}eichtinger and N.~{K}aiblinger.
\newblock {Q}uasi-interpolation in the {F}ourier algebra.
\newblock {\em J. Approx. Theory}, 144(1):103--118, 2007.

\bibitem{feko98}
H.~G. {F}eichtinger and W.~{K}ozek.
\newblock {Q}uantization of {T}{F} lattice-invariant operators on elementary
  {L}{C}{A} groups.
\newblock In H.~G. {F}eichtinger and T.~{S}trohmer, editors, {\em {G}abor
  analysis and algorithms}, {A}ppl. {N}umer. {H}armon. {A}nal., pages 233--266.
  {B}irkh{\"a}user, {B}oston, {M}{A}, 1998.

\bibitem{fest98}
H.~G. {F}eichtinger and T.~{S}trohmer.
\newblock {\em {G}abor {A}nalysis and {A}lgorithms. {T}heory and
  {A}pplications.}
\newblock {B}irkh{\"a}user, {B}oston, 1998.

\bibitem{fest03}
H.~G. {F}eichtinger and T.~{S}trohmer.
\newblock {\em {A}dvances in {G}abor {A}nalysis}.
\newblock {B}irkh{\"a}user, {B}asel, 2003.

\bibitem{fezi98}
H.~G. {F}eichtinger and G.~{Z}immermann.
\newblock {A} {B}anach space of test functions for {G}abor analysis.
\newblock In H.~G. {F}eichtinger and T.~{S}trohmer, editors, {\em {G}abor
  {A}nalysis and {A}lgorithms: {T}heory and {A}pplications}, {A}pplied and
  {N}umerical {H}armonic {A}nalysis, pages 123--170. {B}irkh{\"a}user {B}oston,
  1998.

\bibitem{fi18}
J.~V. {F}ischer.
\newblock {F}our particular cases of the {F}ourier transform.
\newblock {\em {M}athematics}, 12(6):335, 2018.

\bibitem{argi90}
J.~{G}il~de {L}amadrid and L.~N. {A}rgabright.
\newblock {A}lmost periodic measures.
\newblock {\em Mem. Amer. Math. Soc.}, 85(428):vi+219, 1990.

\bibitem{gr01}
K.~{G}r{\"o}chenig.
\newblock {\em {F}oundations of {T}ime-{F}requency {A}nalysis}.
\newblock {A}ppl. {N}umer. {H}armon. {A}nal. {B}irkh{\"a}user, {B}oston,
  {M}{A}, 2001.

\bibitem{grhe99}
K.~{G}r{\"o}chenig and C.~{H}eil.
\newblock {M}odulation spaces and pseudodifferential operators.
\newblock {\em {I}ntegr. {E}qu. {O}per. {T}heory}, 34(4):439--457, 1999.

\bibitem{ho76}
L.~{H}{\"o}rmander.
\newblock {\em {L}inear {P}artial {D}ifferential {O}perators. 4th {P}rinting.}
\newblock {S}pringer, {B}erlin, {H}eidelberg, {N}ew {Y}ork, 1976.

\bibitem{ho89}
W.~{H}{\"o}rmann.
\newblock {\em {G}eneralized {S}tochastic {P}rocesses and {W}igner
  {D}istribution}.
\newblock PhD thesis, {U}niversity of {V}ienna, ({A}{U}{S}{T}{R}{I}{A}), 1989.

\bibitem{ja18}
M.~S. {J}akobsen.
\newblock {O}n a (no longer) {N}ew {S}egal {A}lgebra: a review of the
  {F}eichtinger algebra.
\newblock {\em J. Fourier Anal. Appl.}, 24(6):1579--1660, 2018.

\bibitem{ma09-1}
S.~{M}allat.
\newblock {\em {A} {W}avelet {T}our of {S}ignal {P}rocessing: {T}he {S}parse
  {W}ay}.
\newblock {A}cademic {P}ress, {T}hird edition, 2009.

\bibitem{ma89}
S.~G. {M}allat.
\newblock {M}ultiresolution approximations and wavelet orthonormal bases of
  ${L}^2({R})$.
\newblock {\em Trans. Amer. Math. Soc.}, 315(1):69--87, 1989.

\bibitem{re68}
H.~{R}eiter.
\newblock {\em {C}lassical {H}armonic {A}nalysis and {L}ocally {C}ompact
  {G}roups}.
\newblock {C}larendon {P}ress, {O}xford, 1968.

\bibitem{rest00}
H.~{R}eiter and J.~D. {S}tegeman.
\newblock {\em {C}lassical {H}armonic {A}nalysis and {L}ocally {C}ompact
  {G}roups. 2nd ed.}
\newblock {C}larendon {P}ress, {O}xford, 2000.

\bibitem{sc57}
L.~{S}chwartz.
\newblock {\em {T}h{\'e}orie des {D}istributions. ({D}istribution {T}heory).
  {N}ouveau {T}irage. {V}ols. 1.}
\newblock {P}aris: {H}ermann. xii, 420p., 1957.

\end{thebibliography}

\end{document}